\documentclass{amsart}

\usepackage[english]{babel}
\usepackage[T1]{fontenc}
\usepackage[latin1]{inputenc}
\usepackage{amsmath,amssymb} %alle matteting
\usepackage{amsthm}
\usepackage{amscd}                  %enkle kommutative diagrammer
\usepackage[all]{xy}                %mer avanserte diagrammer
\usepackage{xspace}
\CompileMatrices

\newcommand{\eg}{e.g.,\xspace}
\newcommand{\ie}{i.e.,\xspace}
\newcommand{\prof}{{\,}\widehat{\,}{\,}}

\newcommand{\ess}{\ensuremath{\mathbf{S}}}   
\newcommand{\Z}{\ensuremath{\mathbf{Z}}}   
\newcommand{\Q}{\ensuremath{\mathbf{Q}}} 
\newcommand{\T}{\ensuremath{\mathbf{T}}}    %sirkelgruppen
\newcommand{\A}{\ensuremath{\mathcal{A}}}    %the cartesian square
\newcommand{\bx}{\ensuremath{\mathbf{x}}}    %tuples of finite sets
\newcommand{\I}{\ensuremath{\mathcal{I}}}    %fin sets and injections
\newcommand{\cC}{\ensuremath{\mathcal{C}}}    %a category

\newcommand{\smsh}{\ensuremath{\wedge}}  

\newcommand{\we}{\overset{\sim}{\rightarrow}}

\newcommand{\Lim}[1][{}]{\operatornamewithlimits{lim}_{\overleftarrow{#1}}}
\newcommand{\holim}[1][{}]{\operatornamewithlimits{holim}_{\overleftarrow{#1}
}}
\newcommand{\colim}[1]{\operatornamewithlimits{lim}_{\overrightarrow{#1}}}
\newcommand{\hocolim}[1][{}]{\operatornamewithlimits{holim}_{\overrightarrow{
#1}}}

\newcommand{\CDholim}[1][{}]{\underset{\overleftarrow{#1}}{\operatorname{holi
m}}}

\newcommand{\fib}{\twoheadrightarrow}

\newcommand{\ab}{\ensuremath{\mathcal{A}b}}

\theoremstyle{plain}
\newtheorem{theo}[subsection]{Theorem}
\newtheorem{prop}[subsection]{Proposition}
\newtheorem{lemma}[subsection]{Lemma}
\newtheorem{cor}[subsection]{Corollary}

\theoremstyle{definition}
\newtheorem{Def}[subsection]{Definition}
\newtheorem{ex}[subsection]{Example}
\newtheorem{remark}[subsection]{Remark}

\theoremstyle{remark}

\newcommand{\hofib}{\ensuremath{\operatorname{hofib}}}

\newcommand{\ifi}{\operatorname{ifib}}
\newcommand{\Tor}{\ensuremath{\operatorname{Tor}}}
\usepackage{hyperref}
\begin{document}
\title{Integral excision for K-theory}
\author{Bj{\o}rn Ian Dundas and Harald {\O}yen Kittang}
\address{Department of Mathematics, University of Bergen, Norway}
\email{dundas@math.uib.no}
\address{Mintra AS,
Norway}
\email{harald.kittang@gmail.com}

\maketitle

\numberwithin{equation}{section} 

\begin{abstract}
If $\mathcal A$ is a homotopy cartesian square of
ring spectra satisfying connectivity hypotheses, then the cube
induced by Goodwillie's integral
cyclotomic trace $K(\mathcal A)\to TC(\mathcal A)$
is homotopy cartesian.  In other words, the
homotopy fiber
of the cyclotomic trace satisfies excision.

The method of proof gives as a spin-off new proofs of some old results, as well as some new results, about periodic cyclic homology, and - more relevantly for our current application - the $\T$-Tate spectrum of topological Hochschild homology, where $\T$ is the circle group.
\end{abstract}

\section{Introduction}
\label{sec:intro}

Algebraic K-theory is an important invariant that can be approached
from widely different angles.  There are structural theorems cutting
calculations into smaller, and hopefully more manageable pieces; and
there are approximations by theories that are more open themselves to
calculation.  The aim of this paper is to explain how these two
approaches can be combined in a certain situation.

Algebraic K-theory satisfies the Mayer-Vietoris property for Zariski open
imbeddings of schemes \cite{MR1106918}.  For closed imbeddings this
generally fails, which is bad, for instance if you want to analyze a
singularity where open covers are of little help.  

On the other hand, it is sometimes possible to approximate algebraic
K-theory through the cyclotomic trace $trc\colon
K\to TC$  to topological cyclic homology.  Topological cyclic homology
lacks some of the structural properties of algebraic K-theory, but one
can hope to calculate $TC$ in a given situation.  

This paper proves that the difference between K-theory and
topological cyclic homology, that is, the homotopy fiber of the
cyclotomic trace $\operatorname{hofib}_{trc}$, has the Mayer-Vietoris
property for closed imbeddings.  

The importance of this is that K-theory is wedged in a fiber sequence
$$\operatorname{hofib}_{trc}\to K\to TC$$
where the fiber is structurally accessible and the base functor is
accessible through calculations in stable homotopy theory.
More concretely, this means that, if you have a closed cover,
then algebraic K-theory can be recovered from topological cyclic
homology and the hyper homology of algebraic K-theory with respect to
the closed cover.

When trying to generalize algebraic geometry to ring spectra certain
obstacles are met.  Most successful approaches have focused on
connective  (\ie the negative homotopy groups vanish) ring spectra, and have translated the crucial geometric
invariants through the path component functor $\pi_0$.  Also, the
translation between rings and schemes requires some care.  In
particular, a pushout of affine schemes is in general not an affine
scheme.  When one of the maps involved is a closed embedding things
work out \cite{MR2182775}, and that is the context we are concerned with in this paper.

%More precisely,
\begin{theo}\label{thm:main}
  Let 
\begin{equation*}
   \A=\left\{\begin{CD}
    A^0 @>>> A^1 \\
    @VVV  @VV{f^1}V\\
    A^2 @>>{f^2}> A^{12}
  \end{CD}\right\}
\end{equation*}
be a homotopy cartesian square of connective ring spectra and
$0$-connected maps.  
% %Let $trc\colon K\to TC$ be Goodwillie's
% %integral cyclotomic trace map (see \ref{}).  
% Consider the resulting
% cube
% $$trc_{\mathcal A}\colon K(\mathcal A)\to TC(\mathcal A).$$
%
Then the resulting cube 
$$trc_{\mathcal A}\colon K(\mathcal A)\to TC(\mathcal A)$$ is homotopy cartesian.
\end{theo}
\begin{remark}
  \begin{enumerate}
  \item The topological cyclic homology in question is Goodwillie's
  integral version.  We will recall the necessary details when we need
  them in section \ref{sec:TCandTfixed}.
  \item Theorem \ref{thm:main} says that, under the given connectivity hypotheses, the homotopy
fiber of the cyclotomic trace satisfies excision: it preserves
homotopy cartesian squares.  In the commutative case, the provision that the maps are $0$-connected assures
the connection to geometry: 
$Spec(\pi_0f^j)$ are closed imbeddings, and so affine results are
geometrically interesting.  
Note however, that our ring spectra are not assumed to be commutative.
\item 
It would be desirable to have a statement where just one of the maps,
say $f^1$, were $0$-connected.  With the present line of proof this is
not obtainable, essentially because of a technicality
($\operatorname{Ext}$-completion of infinite sums of torsion modules need not be
torsion), which vanishes under certain finiteness conditions.  We have
refrained from pursuing this issue since it would lengthen the
exposition significantly.
  \end{enumerate}
\end{remark}

\subsection{Notation}
\label{sec:notation}

The category of finite sets and injections is denoted $\I$.  If $X$ is a spectrum, $X\prof$ is its profinite completion and $X_{(0)}$ its rationalization.  If $\mathcal X$ is a cube of spectra, $\ifi\mathcal X$ is the iterated homotopy fiber.  If $M$ is a simplicial abelian group, $HM$ is the associated Eilenberg-Mac Lane spectrum.  The results in this paper are independent of choice of framework for symmetric monoidal smash products, but for concreteness the spectra are supposed to be simplicial functors. 

If $k$ is a natural number, we let $\mathbf k=\{1,\dots,k\}$ and $k_+=\{0,1,\dots,k\}$ considered as a pointed set with base point $0$, and $C_{k+1}$ is the cyclic group of order $k+1$.

\subsection{Side results}
\label{sec:side}

On our way we (re)prove the following results (where $HP$ is periodic
cyclic homology):
\begin{prop}\label{prop:HP}
  \begin{enumerate}
%   \item $HP_n(A)\cong HP_n(A_0)$ when $A=\bigoplus_{k=0}^\infty A_k$ is a graded $\Q$-algebra
% %  \item $HP_n(A)\cong HP_n(\pi_0A)$ when $A$ is a simplicial $\Q$-algebra
  \item \label{prop:HPitem1}If $A\to B$ is a surjection of
    $\Q$-algebras with nilpotent kernel, then the induced map $HP_n(A)\to HP_n(B)$ is
    an isomorphism for every $n$.
  \item \label{prop:HPitem2} Periodic cyclic homology has the Mayer-Vietoris property, in the sense that for a cartesian square
    $\mathcal A$ of $\mathbf Q$-algebras and surjections, there is a long
    exact sequence 
$$\dots\to HP_n(A^0)\to HP_n(A^1)\oplus HP_n(A^2)\to HP_n(A^{12})\to HP_{n-1}(A^{0})\to \dots$$
  \end{enumerate}
\end{prop}

The proofs are very hands-on, filtering cyclic
modules through filtrations where the subquotients are built out of
retracts -- up to multiplication by concrete integers -- of free cyclic
objects (on which periodic homology vanishes).  

The good thing about this is that the proofs are combinatorial enough
to work directly to show vanishing results $\T$-Tate homology of
$THH(-)\prof_{(0)}$, where $THH$ is topological Hochschild
homology.  For instance
\begin{prop}\label{prop:Tatevanish}
  If  $\A$ is a cartesian square of connective $\ess$-algebras and
  $0$-connected maps, then the square
$\left(THH(\A)\prof_{(0)}\right)^{t\T}$ is cartesian.
\end{prop}

\begin{remark}
  The problem of showing the main result without the connectivity
hypothesis on all maps, essentially boils down to the fact that we are not able to
prove that $\left(THH(A)\prof_{(0)}\right)^{t\T}\to
\left(THH(A_0)\prof_{(0)}\right)^{t\T}$ is an equivalence for a graded ring
$A=A_0\oplus A_1\oplus\dots$ without some finiteness hypothesis.
\end{remark}

\subsection{The core of the proof of Theorem 
  \ref{thm:main}}\label{sec:outline} Consider the arithmetic square
$$\begin{CD}
  \ifi\hofib_{trc}(\A)@>>>\ifi\hofib_{trc}(\A)_{(0)}\\
  @VVV@VVV\\
  \ifi\hofib_{trc}(\A)\prof@>>>\ifi\hofib_{trc}(\A)\prof_{(0)}
  \end{CD}.
$$
Theorem \ref{thm:main} claims that $\ifi\hofib_{trc}(\A)\simeq *$, and
so it clearly suffices to show that $\ifi\hofib_{trc}(\A)_{(0)}\simeq\ifi\hofib_{trc}(\A)\prof_{(0)}\simeq *$.

The profinite
completion part, namely that $\ifi\hofib_{trc}(\A)\prof_{(0)}$ is
contractible, is the main result of \cite{DK}, which
relied heavily on the work of Geisser and Hesselholt \cite{GH} in the
discrete ring case, which again used ideas from Corti\~nas' rational paper
\cite{Cor}. 

Suitably reinterpreted, Corti\~nas proved that the
composite 
$$K(\A)_{(0)}\to TC(\A)_{(0)}\to
\left(THH(\A)_{(0)}\right)^{h\T}$$ was cartesian. Corti\~nas formulated
his result in terms of ``negative
cyclic homology'', see see \ref{Def:HPHN} below, but in view of the
equivalence $THH(A)_{(0)}\simeq H(HH(A)\otimes\Q)$ of lemma \ref{lem:THHvsHH},
$\left(THH(\A)_{(0)}\right)^{h\T}$ is just another way of expressing
the cube associated with negative
cyclic homology.  

Hence, to conclude the main theorem, all we have to
do is to prove that 
\begin{lemma}\label{lem:main}  Let 
$\A$
be a homotopy cartesian square of connective ring spectra and
$0$-connected maps. 
Then the resulting cube 
  $$TC(\A)_{(0)}\to
\left(THH(\A)_{(0)}\right)^{h\T}$$ is homotopy cartesian.
\end{lemma}

This follows from the results in section \ref{sec:TCandTfixed}.

\section{Excision and Tate homology}
That rational periodic homology is excisive is well known, and
follows from Cuntz and Quillen's models \cite{CQ}.  However, we need a
proof that is generalizable to a slightly more involved situation.  

In this section we give such a proof.  A very
similar argument gives a simpler proof of Goodwillie's result that
rational periodic homology is insensitive to nilpotent extensions and
only sees the degree zero part of (non-negatively) graded algebras.
As a matter of fact, the way we present it, the results are logically intertwined.

\subsection{Free cyclic objects}
\label{sec:freecyc}
Let $\Delta^o$ and $\Lambda^o$ be the simplicial and cyclic
categories, and let $j\colon\Delta^o\to\Lambda^o$ be the inclusion.
If $X$ is a simplicial object in a category with finite coproducts, we let $j_*X$ be the ``free cyclic
object'' on $X$ (\ie the left Kan
extension associated to the inclusion $j\colon\Delta^o\to\Lambda^o$,
which exists if the category in question has finite coproducts).
Explicitly, the factorization properties of
$\Lambda^o$ (see \eg \cite[6.1.8]{Loday}) give that the $q$-simplices are given by 
$(j_*X)_q=\coprod_{C_{q+1}} X_q$, the coproduct indexed over the
cyclic group $C_{q+1}=\{1,t,t^2,\dots,t^q\}$ with structure maps

\begin{align*}
  d_r(t^s,a)&=
  \begin{cases}
    (t^s,d_{r-s}a)& \text{if $0\leq s\leq r\leq q$ }\\
    (t^{s-1},d_{q+1+r-s}a)& \text{if $0\leq r<s\leq q$}
  \end{cases}\\
s_r(t^s,a)&=
\begin{cases}
  (t^s,s_{r-s}a)& \text{if $0\leq s\leq r\leq q$ }\\
    (t^{s+1},s_{q+1+r-s}a)& \text{if $0\leq r<s\leq q$ }
\end{cases}\\
t(t^s,a)&=(t^{s+1},a),
\end{align*}
where we have written $(t^s,a)$ to signify an ``element'' $a\in X_q$ in the
$t^s$th summand of $(j_*X)_q$.
%NBNB This is a different formulation than Hesselholt by sign on $s$. 

If $Y$ is a cyclic object, the adjoint of the identity is the map
$j_*Y\to Y$ given by $(s,y)\mapsto t^{s}y$.

\begin{ex}\label{ex:Qdef}
 A pointed symmetric monoid $N$ is a symmetric monoid in the symmetric monoidal category of pointed sets and smash products.  The smash product becomes the coproduct in the category of pointed symmetric monoids.  Considering $N$ as a constant simplicial object, the free cyclic object $j_*N$ 
is the cyclic nerve: $(j_*N)_q=N^{\smsh q+1}$ (this is true in general for symmetric monoids).  

The following example of a symmetric pointed monoid will be
important to us shortly: $Q=\{*,0,1\}$ pointed at $*$, with $0+0=0$,
$0+1=1$ and $1+1=*$.  We see that $j_*Q\cong\bigvee_{k=0}^\infty
Q(k)$ where $Q(k)$ is the cyclic subset of $j_*Q$ whose $q$-simplices are either the base point or of
the form $n_0\smsh\dots\smsh n_q$ 
where the sum of the $n$'s is $k$ (so that we have a bijection
$Q(k)_q\cong \left\{n_0,\dots,n_q\in\{0,1\}|\sum n_i=k\right\}_+$). 
\end{ex}

\subsection{Rational retracts of free cyclic objects}
\label{sec:ratret}
We will need a result (Lemma \ref{lem:ratret2} below) about variants of Hochschild homology which naturally are rational retracts of free cyclic objects.  However, we start with a simpler version since in many situations this is all what is needed and it is easier to encode.  In order to highlight certain phenomena we choose an indexation in the simple example which is not the same as the one we fall back on in the general case.  

\begin{Def}
  A cyclic spectrum or simplicial abelian group $Y$ is said to be an {\em almost
    free cyclic object} if there is a simplicial object $X$ and maps $Y\to
  j^*X\to Y$ such that the composite induces multiplication by some integer
  $k\neq 0$ on homotopy $\pi_*Y\to\pi_*Y$.
\end{Def}

If $A$ is a discrete ring, the Hochschild homology $HH(A)$ of $A$ is the
cyclic abelian group $[q]\mapsto A^{\otimes q+1}$ (with tensor products over
the integers unless otherwise noted).  If $A$ is a simplicial ring,
$HH(A)$ is the associated cyclic simplicial abelian group.  Flatness
is always assumed (so really one should take free resolutions, and we
are considering what some people call Shukla homology.  Since all the
applications in this section will be rational and applied to rings that
already may have a simplicial direction, we do not
bother making this explicit).

For a ring $B$ and $B$-bimodule $M$, let $B\ltimes M$ be the square zero
extension of $B$ by $M$.  We have a decomposition 
$$HH(B,M)\cong\oplus_{k\geq0}H(k)(B,M)
$$ of cyclic abelian groups, where $H(k)(B,M)$ consists of the tensors with exactly $k$ factors of $M$ in each
dimension.  

If we set $M(*)=0$, $M(0)=B$, $M(1)=M$, and
$M(n)=\bigotimes_{j=0}^qM(n_j)$ for $n=n_0\smsh\dots\smsh n_q\in
(j_*Q)_q$, where $Q=\{*,0,1\}$ is the pointed symmetric monoid of
example \ref{ex:Qdef}, we get that the group of
$q$-simplices of $H(k)(B,M)$ is isomorphic to
$$\bigoplus_{n\in (Q(k))_q} M(n)$$
where $Q(k)$ is the cyclic subcomplex of $j_*Q$ defined in \ref{ex:Qdef}.
We will use the notation $a/n$ to specify an object
$a=a_0\otimes\dots\otimes a_q$ in the $n=n_0\smsh\dots\smsh n_q$ summand.

The summands with $n_0=1$ (\ie the zeroth factor in the tensor product
$M(n)$ is $M(1)=M$) assemble to a {\em simplicial} subcomplex $G(k)(B,M)\subseteq
H(k)(B,M)$.

If $H$ is a simplicial abelian group, the free cyclic abelian group
$j_*H$ has $q$-simplices
$\bigoplus_{C_{q+1}}H_q%\cong\tilde\Z[C_{q+1}]\otimes H_q
$, and we
write an element $h$ in the $t^j$th summand as $(t^j,h)$.  

\begin{lemma}\label{lem:HjG}
  There is a
cyclic map $$H(k)(B,M)\to j_*G(k)(B,M)$$ given by sending
$a=a_0\otimes\dots\otimes a_q$ in the
$n$'th summand of $H(k)(B,M)_q$
to
  $$\sum_{n_j=1}(t^j,t^{-j}a/t^{-j}n)=\sum_{n_j=1}(t^j,a_j\otimes\dots\otimes
  a_{j-1}/n_j\smsh\dots\smsh n_{j-1}),$$
where the sums are over all $j$ such that $n_j=1$.  
\end{lemma}
\begin{proof}
  To check that this is a well defined cyclic map, let $\phi\in\Delta$,
use the definition of the structure maps in the free cyclic object and
unique factorization $\phi^*t^j=t^{(\phi,j)}\phi_j^*$ to see that the
map commutes with $\phi^*$, basically because the index sets of the
two resulting sums,
$\{i|(\phi^*n)_i=1\}$ and $\{(\phi,j)|n_j=1\}$, are equal.
\end{proof}

For future reference we note
\begin{lemma}\label{lem:ratret}
  The composite 
$$ H(k)(B,M)\to j_*G(k)(B,M)\to H(k)(B,M)
$$ 
is multiplication by $k$, where the first map is defined in Lemma \ref{lem:HjG} and the second is the adjoint of the inclusion.  Hence $H(k)(B,M)$ is an almost free cyclic abelian group. 
\end{lemma}

As an immediate corollary (since rationalization commutes with
infinite coproducts) we get 
\begin{cor}
  The fiber of $HH(B\ltimes M)\to HH(B)$ is rationally a retract of a
  free cyclic object.
\end{cor}
However, our applications are more delicate in that they need to
navigate rather carefully through functors that are not particularly
well behaved with respect to (co)limits, and we will need to refer back to the
precise formulation in Lemma
\ref{lem:ratret} and to the slightly more general Lemma \ref{lem:ratret2} below.  

Let $A=A_0$ be a ring and let $A_1,\dots,A_l$ be $A$-bimodules.  Let $A\ltimes(A_1\oplus\dots\oplus A_l)$ be the square zero extension of $A$. It is convenient to grade this ring, so that $A_j$ is in degree $j$.  

Consider the partitions of $k\geq 0$, \ie sequences $P=\left(k_1\geq
k_2\geq\dots\geq k_r\right)$ of positive integers such that their sum
$k_1+k_2+\dots+k_r$ is $k$ (the empty partition is a partition of $0$).  The {\em length} of $P$ is $r$ and its {\em norm} is 
$|P|=k_1k^{k-1}+k_2k^{k-2}+\dots+k_rk^{k-r}$.  Partitions of $k$ are ordered according to their norm; if $k=4$ we get that $(4)>(3+1)>(2+2)>(2+1+1)>(1+1+1+1)$.

For our purposes it is convenient to use distributivity to decompose the Hochschild homology into cyclic summands:
$$HH(A\ltimes(A_1\oplus\dots\oplus A_l))\cong
\bigoplus_{k\geq 0}\bigoplus_{P}H(P)
$$
where the second summand is over all partitions $P=\left(k_1\geq
k_2\geq\dots\geq k_r\right)$ of $k$, and $H(P)=H(P)(A_0;A_1,\dots A_l)$ is the cyclic abelian group whose group of $q$-simplices is
$$\bigoplus_{f}\bigotimes_{j=0}^qA_{f(j)}
$$
where $f$ varies over the set $S_q(P)$ of functions $C_{q+1}\to l_+$ such that the nonzero values of $f$ correspond to (a permutation of) $P$; \ie  such that there is a bijection $\sigma\colon \mathbf r\to Supp(f)$ with $f(\sigma(j))=k_j$.% $f$ fits in a cartesian square

Let $G(P)$ be the subsimplicial object of $H(P)$ consisting of the summands corresponding to the $f\in S_q(P)$ with $f(0)\neq 0$, and let $H(P)\to j_*G(P)$ be the cyclic map which sends $a$ in the $f\in S_q$ summand to $\sum_{j\in Supp(f)}(t^{f(j)},t^{-f(j)}a)$. 

We note that in the case $B=A$, $M=A_1$, $r=k$, $l=1$, we are in the situation of Lemma \ref{lem:ratret}.  The conclusion holds in the more general context:
\begin{lemma}\label{lem:ratret2}
  Let $A$, $A_1,\dots,A_l$ and $P=\left(k_1\geq\dots\geq k_r\right)$ a
  partition of $k>0$.  The map $H(P)\to j_*G(P)$ is well defined, and
  the composite $$H(P)\to j_*G(P)\to j_*H(P)\to H(P)$$ is multiplication
  by the length $r$ of $P$, and so $H(P)=H(P)(A;A_1,\dots,A_l)$ is an almost free cyclic object.
\end{lemma}

Eventually this leads to the lemma that decomposes relative Hochschild
homology in terms of almost free cyclic objects.

If $A\fib A/I$ is a surjection of flat (= flat in every degree) simplical rings, let $F^k(A,I)=F^k$ be the cyclic subobject of
$HH(A)$ which in degree $q$ is given by
$$F^k_q=\sum_{\sum{n_j}\geq k}\otimes_{j=0}^qI^{n_j}.$$
We get that $F^0=HH(A)$ and $F^0/F^1=HH(A/I)$.  

\begin{lemma}
  Let $A\fib A/I$ be a surjection of flat simplical rings.  Then,
  for each $k>0$ there is a sequence of surjections
  $$F^k/F^{k+1}\fib X^k(1)\fib\dots\fib X^k({p(k)})=0,$$
  where $p(k)$ is the number of partitions of $k$ and such that the kernel of
  each surjection is an almost free cyclic object.
\end{lemma}

 \begin{proof}

There is a natural isomorphism $F^k/F^{k+1}(A,I)\cong
F^k/F^{k+1}(gr(A,I))$, where $gr(A,I)$ is the associated graded pair %with
$\left(\bigoplus_{j=0}^\infty I^j/I^{j+1},\bigoplus_{j=1}^\infty I^j/I^{j+1}\right)$, and so  we
only need to worry about the graded situation, where $A=\bigoplus_{n=0}^\infty A_n$
and $I=\bigoplus_{n=1}^\infty A_n$.
We may assume  that for each $n\geq 0$ the $n$'th homogenous piece $A_n$ is (degreewise) flat.  Then $HH(A)$ splits as a sum according to total
  degree.  The piece of total degree $0$ is simply $HH(A_0)$. The group of $q$-simplices in $F^k/F^{k+1}$ is isomorphic to
$\bigoplus\bigotimes_{j=0}^\infty A_{n_j}$ where the sum is over
sequences of non-negative integers $n_0,\dots,n_q$ such that $\sum n_j=k$.

Given a partition $P=\left(k_1\geq k_2\geq\dots\geq k_r\right)$ of $k$, the group of $q$-simplices in the cyclic abelian group $H(P)(A_0;A_1,\dots A_k)$ discussed before Lemma
\ref{lem:ratret2} is a subgroup of the group of $q$-simplices in
$F^k/F^{k+1}$, but does not usually form a subcomplex as $q$ varies.
Actually, the group of $q$-simplices in $F^k/F^{k+1}$ is isomorphic to
$\bigoplus H(P)(A_0;A_1,\dots A_k)_q$, where the sum runs over all partitions $P$ of $k$, but the face maps can take summands belonging to a certain partition to a summand belonging to a smaller partition.

However, if $P_1>P_2>\dots >P_{p(k)}$ are all the partitions of $k$, we get that $H(P_1)(A_0;A_1,\dots
A_k)=H(k)(A_0,A_k)$ (in the notation of Lemma \ref{lem:ratret}) is a
cyclic subobject of $F^k/F^{k+1}$.  Let $X^k(1)$ be the quotient of $H(k)(A_0,A_k)\to F^k/F^{k+1}$, and notice
that $H(P_2)(A_0;A_1,\dots A_k)$ is a cyclic subobject.  Calling the
quotient of this inclusion $X^k(2)$, we notice that $H(P_3)(A_0;A_1,\dots A_k)$ is a
cyclic subobject, and so on, until we reach $X^k(p(k))=0$.  By Lemma
\ref{lem:ratret2}, all the kernels in the sequence of surjections
$$F^k/F^{k+1}\fib X^k(1)\fib\dots\fib X^k(p(k))=0$$ are almost free
cyclic abelian groups.
\end{proof}

\subsection{Homology and free cyclic objects}There is another view on free cyclic objects in a category $\cC$ with coproducts which is useful for some purposes.  For convenience, if $X$ is an object in $\cC$ and $S$ is a finite set, we write $X\otimes S$ for the $S$-fold coproduct of $X$ with itself.

Recall that if $I$ is a small category, $\cC$ a category with coproducts and $M\colon I^o\times I\to\cC$ we can define the (Hochschild) homology $H(I,M)$ as the simplicial object in $\cC$ whose $n$-simplices is given by $\coprod_{i_0,\dots,i_n\in I} M(i_0,i_n)\otimes I(i_1,i_0)\otimes\dots\otimes I(i_n,i_{n-1})$ with face maps given by composition and the functoriality of $M$ and degeneracies by inserting identity maps.  
If $M\colon J^o\times J\to\cC$, then $f\colon I\to J$ induces an obvious map $f\colon H(I,f^*M)\to H(J,M)$.
If $M$ factors as $N\circ pr$ where $pr$ is the projection $I^o\times
I\to I$ one most frequently refers to $H(I,M)$ as the (simplicial
replacement of the) homotopy colimit of $N$.

If $\cC$ has coequalizers we let $H_0(I,M)$ be the coequalizer of the
two face maps from the $1$-simplices to the $0$-simplices.

If $f\colon I\to J$ and $X\colon I\to \cC$ are functors, we can
identify the left Kan extension $(f_*X)(j)$ with the homology
$H_0(I,X(-)\otimes J(f(-),j))$, and 
$$ho(f_*)X=H(I,X(-)\otimes J(f(-),j))$$ is a ``homotopy left Kan
extension''.

In the particular case where $f=id\colon I=I$,
the map $$ho(id_*)X(i)=H(I,X(-)\otimes I(-,i))\to X(i)$$ given by composition has a
simplicial contraction given by inserting identities, and so we have a
homotopy version of the dual Yoneda lemma (which says that
$(id)_*X\cong X$).

Recall the inclusion $j\colon\Delta^o\subseteq\Lambda^o$.
\begin{lemma}
Let $M$ be a simplicial object in a category with finite colimits.
Then $ho(j_*)M\to j_*M$ is an objectwise simplicial homotopy
equivalence, in the sense that for each $[q]\in\Lambda^o$, the map of
simplicial objects (the target is constant)
$ho(j_*)M([q])=H(\Delta^o,M\otimes \Lambda^o(j(-),[q])))\to (j_*M)_q$ is a simplicial homotopy equivalence.
\end{lemma}
\begin{proof}
The canonical factorization in $\Lambda$, \cite[6.1.8]{Loday}, gives rise to a factorization of the identity $$\Lambda([s],[t])\cong \Delta([s],[t])\times Aut_{\Lambda}([t])\to Aut_{\Lambda}([s])\times \Delta([s],[t])\to \Lambda([s],[t]),$$ where the latter function is the composition in $\Lambda$.  By uniqueness of the factorization this can be promoted to a split monomorphism
$$H(\Delta^o,M\otimes \Lambda^o(j(-),[q])))\to H(\Delta^o,M\otimes \Delta^o(-,[q])))\otimes Aut_{\Lambda}([q]),$$
which induces the isomorphism $(j_*M)([q])\cong M_q\otimes
Aut_{\Lambda}([q])$ discussed earlier on the zeroth homology.  
In effect the map $ho(j_*)M([q])\to (j_*M)_q$ becomes a retract of
$ho(id_*)M([q])\otimes Aut_\Lambda([q])\to (id_*)M([q])\otimes
Aut_\Lambda([q])\cong M([q])\otimes Aut_\Lambda([q])$ which is a
simplicial homotopy equivalence by the homotopical dual Yoneda lemma.
\end{proof}

As an example, if $M$ is a cyclic module, \ie a functor $\Lambda^o\to\ab$, then $HC(M)=H(\Lambda^o,M)$ and $HH(M)=H(\Delta^o,j^*M)\simeq j^*M$, and $j\colon\Delta\to\Lambda$ induces a map $HH(M)\to HC(M)$.  In the special case of a free cyclic module one has
\begin{lemma}\label{lem:HHtoHCforfree}
  Let $M$ be a simplicial abelian group.  Then the map $HH(j_*M)\to HC(j_*M)$ is a split surjection in the homotopy category.
\end{lemma}
\begin{proof}
We will prove that the corresponding statement is always true for the homotopy Kan extension.  As we have seen, the homotopy and categorical notions coincide up to homotopy for $j\colon\Delta^o\to\Lambda^o$, so this proves the result.

  Consider the general situation $f\colon I\to J$ and $X\colon I\to \cC$.
We prove that the map 
$$H(I,f^*ho(f_*)X)\to H(J,ho(f_*)X)$$ induced by $f$ is a split epimorphism modulo simplicial homotopy.

Consider the inclusion 
$$X(i)\to f^*ho(f_*)X(i)_n=
\coprod_{i_0\gets\dots\gets i_n,\, f(i_n)\gets f(i)}X(i_n)$$ onto the $i=\dots=i,\, f(i)=f(i)$ summand.  This gives a natural transformation $X\to f^*ho(f_*)X$.
Precomposing the map we want to show is a split epimorphism with
$H(I,X)\to H(I,f^*ho(f_*)X)$ gives us a map $F\colon H(I,X)\to
H(J,ho(f_*)X)$.  The claim will therefore follow once we show that $F$
is simplicially homotopic to a
simplicial homotopy equivalence $G$.

Now, $F$ sends $a=x\otimes(i_0\gets\dots\gets i_n)$ to $F(a)=((x\otimes 1)\otimes (i_n=\dots=i_n))\otimes (f(i_0)\gets\dots\gets f(i_n))$.  Letting $k$ vary from $0$ to $n$, the assignments sending $a$ to $((x\otimes 1)\otimes (i_k=\dots=i_{k}\gets\dots\gets i_n))\otimes (f(i_0)\gets\dots\gets f(i_k)=\dots= f(i_k))$ assemble to a simplicial homotopy between $F$ and $G$, where $G(a)=((x\otimes 1)\otimes (i_0\gets\dots\gets i_n))\otimes (f(i_0)=\dots =f(i_0))$.   

The inclusion $X(i)\to H(J,X(i)\otimes J(f(i'),-))_n=\coprod_{j_0\gets\dots\gets j_n,\, j_n\gets f(i')}X(i)$ onto the $f(i')=\dots=f(i'),\,f(i')=f(i')$ summand gives a natural transformation.
  The map $G$ is a composite
$$H(I,X)\to H(I,(i',i)\mapsto H(J,X(i)\otimes J(f(i'),-)))\cong H(J,H(I,X\otimes J(f(-),-))),$$
where the first map is induced by the degeneracy $X(i)\to H(J,X(i)\otimes J(f(i'),-))$ (which is a simplicial homotopy equivalence) and the isomorphism is simply reversal of priorities.

The lemma is the special case where $I=\Delta^o$, $J=\Lambda^o$, $X=M$ and $f=j\colon I\to J$.
\end{proof}

\subsection{Periodic cyclic homology}
\label{sec:HP}

In order to fix notation and for reference we
recall the construction of (periodic) cyclic homology, see for instance \cite{Loday} for more details.  Let
$M\colon\Lambda^o\to\ab$ be a cyclic abelian group, and define the
periodic bicomplex $CP(M)$
$$
\begin{CD}
@.@.@.@.\\
@.@VVV@VVV@VVV@.\\
  \dots @<{1+t}<< M_3@<{1-t+t^2-t^3}<<M_3@<{1+t}<<M_2@<{1-t+t^2-t^3}<<\dots\\
  @.@V{-d_2+d_1-d_0}VV@V{d_0-d_1+d_2-d_3}VV@V{-d_2+d_1-d_0}VV@.\\
  \dots@<{1-t}<<M_2@<{1+t+t^2}<<M_2@<{1-t}<<M_2@<{1+t+t^2}<<\dots\\
  @.@V{d_1-d_0}VV@V{d_0-d_1+d_2}VV@V{d_1-d_0}VV@.\\
  \dots@<{1+t}<<M_1@<{1-t}<<M_1@<{1+t}<<M_1@<{1-t}<<\dots\\
  @.@V{-d_0}VV@V{d_0-d_1}VV@V{-d_0}VV@.\\
  \dots@<{1-t=0}<<M_0@=M_0@<{1-t=0}<<M_0@=\dots
\end{CD}
$$
repeated indefinitely in both horizontal directions, with the middle column (which is the Moore complex of the simplicial
abelian group underlying $M$) in degree $0$. The odd columns are acyclic. Notice that the rows are
acyclic when $M$ is rational.

The homology groups of the zero'th column are referred to as {\em Hochschild homology} $HH_*(M)$, and are naturally isomorphic to the
homotopy groups
$\pi_*(j^*M)$ where $j^*$ is precomposition with $j\colon\Delta\to\Lambda$, see the previous section.  

The homology of the
total complex consisting of the non-negative columns only is referred to as
{\em cyclic homology}, $HC_*(M)$, and can alternatively be calculated as the
homotopy groups of $\hocolim[\Lambda^o]M=H(\Lambda^o,M)$.

\begin{Def}\label{Def:HPHN}
The {\em periodic homology} $HP_*(M)$ of $M$ is the homology of the total complex
$\{n\mapsto\prod_{r+s}CP_{(r,s)=n}(M)\}$.  {\em Negative cyclic
  homology} $HC^-(M)$  is
the homology of the total complex of the sub bicomplex
$CC^-(M)\subseteq CP(M)$ concentrated in non-positive degrees.  
\end{Def}
We get
canonical isomorphisms $HC_{*-2}(M)\cong H_*(CP(M)/CC^-(M))$ and 
long exact sequences 
$$\minCDarrowwidth0cm
\begin{CD}
  \dots@>>>HC_{n-1}(M)@>>>HC^-_{n}(M)@>>>HP_n(M)@>>>HC_{n-2}(M)@>>>\dots\\
  @.@|@VVV@VVV@|@.\\
  \dots@>>>HC_{n-1}(M)@>B>>HH_{n}(M)@>>>HC_n(M)@>S>>HC_{n-2}(M)@>>>\dots
\end{CD}.
$$

\begin{lemma}\label{lem:HPfreevanish}
  If $N$ is a simplicial abelian group, then $HP(j_*N)=0$, and so
  $HC_{n-1}(j_*N)\to HC^-_{n}(j_*N)$ is an isomorphism for all $n$.
\end{lemma}
\begin{proof}
  The map $HH_n(j_*N)\to HC_n(j_*N)$ is split surjective by Lemma \ref{lem:HHtoHCforfree}.
  Hence the map $S\colon HC_n(j_*N)\to HC_{n-2}(j_*N)$ is zero.
Filtering $CP(M)$ by columns, we get 
%that $CP(M)=\Lim[n]CC(M)[2n]$, giving 
the short exact sequence
$$0\to\Lim[S]{}^{1}HC_{n-2k+1}(M)\to HP_n(M)\to \Lim[S]HC_{n-2k}(M)\to
0,$$
and so $HC_*(j_*N)=0$.
\end{proof}

\subsection{Consequences for functors vanishing on almost free cyclic objects}
\label{sec:cons-tate-peri}
The fact \ref{lem:HPfreevanish} that periodic homology vanishes on free cyclic objects, and
the retracts of Lemma \ref{lem:ratret} lead to a sequence of important
results.  

Recall the following result by Goodwillie from \cite[p. 356]{Grel}.
We repeat it here since we need extra information which is obvious
from Goodwillie's proof, but not stated as part of his result.

\begin{lemma}\label{lem:good}
  Suppose $I\subseteq A$ is a (k-1)-connected ideal in a simplicial
  ring.  Then there exist a degreewise free simplicial ring $F$ and a
  $k$-reduced (\ie $J_q=0$ for
$q< k$) ideal $J\subseteq F$ generated in each degree by
  generators of $F$, and an equivalence of surjections of simplicial rings
$$
\begin{CD}
  F@>>> F/J\\@V{\simeq}VV@V{\simeq}VV\\A@>>> A/I
\end{CD}.
$$
\end{lemma}

The conditions on the functor $X$ in the following proposition are satisfied for
the Eilenberg-MacLane spectrum associated with periodic homology of
rational algebras, and
so the statement \ref{prop:HPitem1} in Proposition \ref{prop:HP} about
nilpotent extensions follows.
\begin{prop}\label{prop:Xvanish}
  Let $X$ be a pointed homotopy functor
  from cyclic simplicial abelian groups to spectra
  satisfying the homotopy properties
\begin{enumerate}
\item $X$ preserves finite homotopy limits,
\item if $\dots\to F^3\to F^2\to F^1$ is a sequence of cyclic simplicial abelian
  groups such that the connectivity of $F^n$ goes to infinity as $n$
  goes to infinity, then $\holim[n]X(F^n)\simeq *$, and
\item $X$ vanishes on almost free cyclic objects.
\end{enumerate}
  Assume that $A\to B$ is a map of simplicial rings and (at least) one of the
  following conditions are met: 
  \begin{enumerate}
  \item \label{prop:Xvanishitem1} $A\to B$ is a surjection of flat simplicial rings with
    nilpotent kernel.
  \item \label{prop:Xvanishitem2} $A\to B$ is a $1$-connected map simplicial rings.
  \end{enumerate}
Then
$$X\,HH(A)\to X\,HH(B)$$
is an equivalence.
\end{prop}
\begin{proof}
First, assume that $A\to B$ is a surjection of flat rings with
kernel $I$ satisfying $I^n=0$.
Recall the filtration of $HH(A)$ given just before Lemma
\ref{lem:ratret2}. 
 Let $F^k(A,I)=F^k$ be the simplicial subcomplex of
$HH(A)$ which in degree $q$ is given by
$F^k_q=\sum_{\sum{n_j}\geq k}\otimes_{j=0}^qI^{n_j}.$  From Lemma
\ref{lem:ratret2} and the conditions on $X$ we get that
$X(F^k/F^{k+1})\simeq *$ for all $k>0$, and so $X(F^1)\simeq
X(F^2)\simeq\dots\simeq\holim[k]X(F^k)%\simeq X(\holim[k]F^k)
$.  Hence,
in order to prove that $X\,HH(A)\to X\,HH(B)$ is an equivalence, we
only need to show that 
the connectivity of $F^k$ grows to infintiy with $k$,
which follows since $F^k(A,I)_q=0$
for $k\geq n(q+1)$.

Now, let $A\to B$ be a $1$-connected map.  Since $X$ is a homotopy functor
one may assume that the map is a surjection of flat simplicial
rings and by Lemma \ref{lem:good} that the kernel $I$ is
$1$-reduced (that is, the group of zero simplices is trivial:
$I_0=0$).  Let $A(1)=A$ and $I(1)=I$.  We will construct a sequence of
ring-ideal pairs 
$$\dots\to (A(n),I(n))
\to\dots\to
(A(2),I(2))\to (A(1),I(1))$$ such that for each $n$ the following is true
\begin{enumerate}
\item for each $[q]\in\Delta^o$ the ring $A(n)_q$ is free and the
  ideal $I(n)_q$ is
generated as an ideal by generators of $A(n)_q$
\item the map $A(n+1)\to A(n)$ is an equivalence and $I(n+1)\to I(n)$
  factors as $I(n+1)\to I(n)^2\subseteq I(n)$ with the first map an equivalence, and
\item $I(n)$ is $n$-reduced.
\end{enumerate}
 Assuming
that for given $n$ the pair $(A(n),I(n))$ is already constructed, we consider
$I(n)^2$.  Since $I(n)_q$ is generated by generators of $A(n)_q$, both
$A(n)/I(n)$ and $A(n)/I(n)^2$ are degreewise flat.  Since $I(n)$ is $n$-reduced, the short exact
sequence 
$$
\begin{CD}
  0\to\ker\{\text{mult.}\}\to I(n)\otimes I(n)@>{\text{mult.}}>> I(n)^2\to 0
\end{CD}
$$ gives that
$I(n)^2$ is $n$-connected, and we let the equivalence $(A(n+1),I(n+1))\to
(A(n),I(n)^2)$ be the result of Lemma \ref{lem:good}, replacing an
$n$-connected ideal by an $n+1$-reduced ones.  

Since $I(n)$ is
$n$-reduced, the homotopy fiber $F(n)$ of $$HH(A(n))\to HH(A(n)/I(n))$$
is $n-1$-connected.   Letting $G(n)$ be the homotopy fiber of
$$HH(A(n))\to HH(A(n)/I(n)^2)$$ we see that $F(n+1)\to F(n)$ factors as
$F(n+1)\we G(n)\to F(n)$.  By the first part of the proposition (regarding
nilpotent extensions), the map $X(G(n))\to X(F(n))$
is an equivalence.    Consequently
the homotopy fiber $X(F(1))$ of $X\,HH(A)\to X\,HH(A/I)$ is equivalent to
$\holim[n] X\,F(n)$, and as the connectivity
of $F(n)$ grows to infinity with $n$, our assumptions about the
functor $X$ implies that $\holim[n] X\,F(n)$ is contractible.
\end{proof}

\begin{Def}
  A {\em split square} of simplicial
rings is a categorically cartesian square of simplicial
flat rings, where all maps are split surjective. 
\end{Def}  
If $\A$ is a commutative square of simplicial flat rings and split surjections, set $A^{12}=I(0)$, $I(1)=\ker\{f^1\}$
and $I(2)=\ker\{f^2\}$.  
That the square is categorically cartesian is then the same as the condition that
the intersection $I(1)\cap I(2)$ is trivial.

In this situation, the iterated fiber of $HH(\A)$ is, via distributivity, isomorphic to the
cyclic abelian group with $q$-simplices
$$\bigoplus_{f}\bigotimes_{i=0}^qI(f(i))$$
where the sum is over all functions $f\colon\Z/(q+1)\to\Z/3$ (not necessarily linear) with both
$f^{-1}(1)$ and $f^{-1}(2)$ non-empty.
\begin{Def}\label{Def:Af}
  Given a function
$f\colon\Z/(q+1)\to\Z/3$, let $A_f$ be the set consisting of the $j$
  in $\Z/(q+1)$ such
that $f(j)=2$ and such that there is an $i$ 
with $f(i)=1$ and such that all intermediate values of $f$ (in cyclic
ordering from $i$ to $j$) are $0$.  
\end{Def}
\begin{ex}
  If $f,g\colon\Z/11\to\Z/3$ have values
$$
\begin{tabular}{c|llllllllllc}
  $n$   &$0$&$1$&$2$&$3$&$4$&$5$&$6$&$7$&$8$&$9$&$10$\\\hline
  $f(n)$&$2$&$2$&$0$&$1$&$2$&$1$&$1$&$0$&$2$&$0$&$1$\\
  $g(n)$&$2$&$2$&$0$&$1$&$2$&$1$&$1$&$0$&$2$&$0$&$2$
\end{tabular}
$$
then $A_f=\{0,4,8\}$ and $A_g=\{4,8\}$
\end{ex}

\begin{lemma}\label{lem:XPA}
 For a simplicial ring $A$, let $P(A)=HH(A)_{(0)}$ or $P(A)=HH(A)\prof_{(0)}$.
  Let $X$ be a homotopy functor from cyclic groups to spectra,
  preserving homotopy limits and vanishing on free cyclic objects.
  Then $XP(\A)$ is cartesian, where $\A$ is a cartesian square of
  simplicial rings and $0$-connected maps.
\end{lemma}
\begin{proof}
Let $\A$ be a split square.  Note that, since $I(1)\cdot I(2)\subseteq I(1)\cap I(2)=0$, we have a decomposition
of the iterated fiber of $HH(\A)$ into a sum $\bigoplus_{k=1}^\infty
H(k)$ where $H(k)$ is the cyclic abelian group with $q$-simplices
$$H(k)_q=\bigoplus_{\substack{f\\|A_f|=k}}\bigotimes_{i=0}^qI(f(i)).$$

Analogous to the argument in Lemma \ref{lem:ratret} there is an
interesting subsimplicial abelian group $G(k)\subseteq H(k)$ given
as the sum over only those $f$ with $|A_f|=k$ and $0\in A_f$, and a map
$$H(k)\to j_*G(k).$$
sending $a\in H(k)_q$ in the $f$th summand to $\sum_{r\in
  A_f}(r,t^{-r}a)$.
Notice that the composite $H(k)\to j_*G(k)\to H(k)$ is multiplication
by $k$, and so $H(k)$ is almost free cyclic.  This proves the lemma in
  the case where the square $\A$ is split since the connectivity of
  $H(k)$ goes to infinity with $k$ and so
  $\bigoplus_{k>0}H(k)\simeq \prod_{k>0}H(k)$ is a retract of
  a free cyclic object both under rationalization and under profinite
  completion followed by rationalization.  

We reduce the general case to the split case.
For simplicity of notation let
$$\A=\left\{
  \begin{CD}
    A@>>>B\\@VVV@VVV\\C@>g>>D
  \end{CD}
  \right\}
$$
with $B\to D$ and $C\to D$ surjective on $\pi_0$.  We may assume that
these maps are fibrations, and so surjections (since a map $B\to D$ of simplicial abelian
groups is a fibration iff $B\to D\times_{\pi_0D}\pi_0B$ is a
surjection) and that the square is categorically cartesian.  

Consider the (bi)simplicial resolution of $D$ 
$$B^D=\{r\mapsto B\times_D\dots\times_D B
\}$$ ($r+1$ factors of $B$ in degree $r$ and multiplication componentwise) where $d_i$ projects away from the
$i$'th factor and $s_i$ repeats the $i$'th factor.  That $B^D\to D$ is
an equivalence is fairly general, but in this context can be seen
directly by noting that the normal complex of $B^D$ is simply the
inclusion of $0\times_DB$ into $B$.

By taking
pullback, we get a resolution of $\A$ with $r$-simplices
$$
\begin{CD}
  B\times_DB^D_r\times_DC@>>>B\times_DB^D_r\\
  @VVV@VVV\\
  B^D_r\times_DC@>>>B^D_r
\end{CD}.
$$
Note that $B\times_DB^D$ and $B\times_DB^D\times_DC$ have an ``extra
degeneracy'' given by duplicating the first factor:
$(b,b_0,\dots,b_r,c)\mapsto (b,b,b_0,\dots,b_r,c)$.  

If $i\colon\{1,2\}\to\{0,\dots,s\}$ is an injection and $t\in
\{0,\dots,s\}$, let 
$I(i,t)$ equal $B\times_DB^D\times_DC$ if $t\notin im(i)$ and
$I(i,i(1))$ (resp. $I(i,i(2))$) be the ideal $0\times_DC$
(resp. $B\times_D0$) in $B\times_DB^D\times_DC$.

Applying
Hochschild homology to the square in each dimension and taking the
iterated kernel gives us a simplicial  cyclic object which in dimension $(r,s)$ is
$$I_{rs}=\sum_{i}\bigotimes_{t=0}^sI_r(i,t)\subseteq
(B\times_DB^D_r\times_DC)^{\otimes s+1}.$$
Note that the extra degeneracy $B\times_DB^D_r\times_DC\to
B\times_DB^D_{r+1}\times_DC$ induces a map on all the $I_r(i,t)$'s
compatible with the structure map in the Hochschild direction, giving
us a simplicial cyclic object 
$I=\left\{[r]\mapsto I_{r}=\{[s]\mapsto I_{rs}\}\right\}$
and a {\em simplicial} homotopy equivalence $I\we I_{-1}=\ifi HH(\A)$.

Simplicial homotopy equivalences are preserved when functors are applied degreewise to them, and so we get a simplicial homotopy equivalence
$$\{[r]\mapsto X(I_r)\}\we X(I_{-1}).
$$
But since $X$ preserves cartesian squares $X(I_r)$ is the iterated
fiber of $X\circ HH$ applied to the $r$-simplices of our resolution of
$\A$.  In dimension $r$ this splits in the vertical direction, so it
is enough to show excision in cartesian squares with vertical (or horizontal)
splittings.

e may repeat the argument above, starting this time with a square with horizontal
splitting we reduce to the case where both the vertical and the
horizontal maps split.

\end{proof}
Note that we did not assume that $X$ could be ``calculated degreewise'' (which is false in the applications we have in mind), but got around this by considering simplicial homotopy equivalences, where we could apply $X$ degreewise to our resolution without destroying the homotopy type in our special case.

\subsection{Proof of Proposition \ref{prop:HP} and \ref{prop:Tatevanish}}
\label{sec:TateHPvanish}
\begin{proof}[Proof of Proposition \ref{prop:HP}]
Let $X$ be the Eilenberg-MacLane
spectrum associated with periodic cyclic homology and apply
Proposition \ref{prop:Xvanish} and the $P(A)=HH(A)_{(0)}$ part of
Lemma \ref{lem:XPA} (rationalization doesn't change anything since the
rings were already rational).
\end{proof}

\begin{proof}[Proof of Proposition \ref{prop:Tatevanish}]
By resolving connective $\ess$-algebras by simplicial rings as in
\cite{D97}, we see that it is enough to establish
\ref{prop:Tatevanish} for $\A$ a cartesian square of simplicial rings,
with all maps $0$-connected.
In Lemma \ref{lem:XPA}, let $P(A)=HH(A)\prof_{(0)}$.  By Lemma
\ref{lem:THHvsHH} below, the Eilenberg-MacLane spectrum associated
with  $P(A)$ is equivalent to $THH(A)\prof_{(0)}$.  Let $X(M)=(H(M))^{t\T}$ be
the $\T$-Tate homology of the Eilenberg-MacLane spectrum, and observe
that by Lemma \ref{lem:Tateoffreevanish} below, $X$ satisfies the
conditions of Lemma \ref{lem:XPA}, showing that
$(THH(\A)\prof_{(0)})^{t\T}$ is cartesian.
\end{proof}

\begin{Def}
  Let $X$ be a spectrum and let $N\colon\Z\to\Z_+$ be a function from the integers to the positive integers.  We say that $X$ is {\em $N$-annihilated} if for each $k$ the group $\pi_kX$ is annihilated by $N(k)$.  A map $X\to Y$ is an {\em $N$-equivalence} if its homotopy fiber is $N$-annihilated, and a {\em torsion equivalence} if it is an $M$-equivalence for some unspecified $M\colon\Z\to\Z_+$.
\end{Def}
Note that there is no finiteness requirements in this definition, just a statement about the torsion.
\begin{lemma}\label{lem:THHvsHH}
  Let $A$ be a simplicial ring.  Then the linearization map
$$THH(HA)\to H(HH(A))$$
is a torsion equivalence.  Consequently there are a natural equivalences of
cyclic spectra
\begin{align*}
  THH(A)_{(0)}\we &H\left(HH(A)\right)_{(0)}\\
  THH(A)\prof_{(0)}\we &H\left(HH(A)\right)\prof_{(0)}.
\end{align*}
\end{lemma}
\begin{proof}
If a map of simplicial spectra is a degreewise torsion equivalence then its diagonal is a torsion equivalence.
  The topological Hochschild homology of $HA$ is a simplicial spectrum which in  dimension $q$ is equivalent to $HA\smsh^L_{\ess}\dots\smsh^L_{\ess}HA$ and maps to  $HA\smsh^L_{H\Z}\dots\smsh^L_{H\Z}HA$ which is equivalent to the $q$-simplices of $H(HH(A))$.  Hence, it is enough to show that for simplicial abelian groups $M$ and $N$ the map $HM\smsh^L_{\ess}HN\to HM\smsh^L_{H\Z}HN$ a torsion equivalence.  There is an associated map of first quadrant spectral sequences with $E^2$-sheet
$$\Tor_*^{\pi_*\ess}(\pi_*M,\pi_*N)\to \Tor_*^{\Z}(\pi_*M,\pi_*N)$$
converging to $\pi_*(HM\smsh^L_{\ess}HN)\to \pi_*(HM\smsh^L_{H\Z}HN)$.  Now, the map of $E^2$ sheets has kernel and cokernel with annihilated by integers depending on position since $\ess\to H\Z$ is a torsion equivalence.  The numbers annihilating the kernel and cokernel do not change as we move to the $E^\infty$-sheets, and moving to $\pi_k(HM\smsh^L_{\ess}HN)\to \pi_k(HM\smsh^L_{H\Z}HN)$ the kernel and cokernel are annihilated by the product of the numbers needed for the $E^\infty_{s,k-s}$ as $s$ runs from $0$ to $k$.
\end{proof}

\begin{cor}\label{cor:tosionorbit}
  There is a function $L\colon\Z\to\Z_+$ such that, for any subgroup $C$ of the circle, the map
$$|THH(HA)|_{hC}\to |H(HH(A))|_{hC}$$
is an $L$-equivalence.
\end{cor}
The point of this corollary is that $L$ does not depend on $C$.
\begin{proof}
  Consider the spectral sequence calculating the $C$-homotopy orbits of the homotopy fiber $F$ of $|THH(HA)|\to |H(HH(A))|$.  Lemma \ref{lem:THHvsHH} gives that $F$ is $N$-annihilated by some $N$.  Hence $E^1_{s,r}=\pi_sF$ and $E^\infty_{r,s}$ are annihilated by $N(s)$ and $\pi_nF_{hC}$ is annihilated by $L(n)=N(0)\cdot N(1)\cdot\dots\cdot N(n)$.
\end{proof}

\begin{lemma}\label{lem:Tateoffreevanish}
  Let $Y$ be a simplicial spectrum.  Then the $\T$-Tate homology of
  $|j_*Y|$ vanishes.  
\end{lemma}
\begin{proof}
  This follows since $|j_*Y|\cong \T_+\smsh |Y|$, and Tate homology
  vanishes on free objects. 
\end{proof}

\begin{cor}\label{cor:retroffree}
  Let $X$ be an almost free cyclic spectrum.  Then the natural map
  $(X^{h\T})_{(0)}\to (X_{(0)})^{h\T}$ is an equivalence.
\end{cor}
\begin{proof}
  By the lemma, both the source and the target of $(X^{t\T})_{(0)}\to
  (X_{(0)})^{t\T}$ are contractible, so the $\T$-norm maps $S^1\smsh
  (X_{h\T})_{(0)}\to (X^{h\T})_{(0)}$ and
  $S^1\smsh(X_{(0)})_{h\T}\to(X_{(0)})^{h\T}$ are both equivalences.
  Homotopy orbits commute with rationalization, so we are done.
\end{proof}

\section{Relations between $TC$ and homotopy $\T$-fixed points}
\label{sec:TCandTfixed}

Topological cyclic homology $TC(A)$ of a connective $\ess$-algebra $A$
is most effectively defined integrally, as in \cite{DGM}, by a cartesian square
$$
\begin{CD}
  TC(A)@>>>THH(A)^{h\T}\\
@VVV@VVV\\
\left(\CDholim[R,F]THH(A)^{C_n}\right)\prof@>>>\left(\CDholim[F]THH(A)^{hC_n}\right)\prof
\end{CD}.
$$
Here $R$ and $F$ are maps $THH(A)^{C_{mn}}\to THH(A)^{C_{n}}$ called
respectively the restriction and Frobenius (the latter is just
inclusion of fixed points) where $m$ and $n$ are
positive integers.  The homotopy limit in the lower left corner is
over the category whose objects are the positive integers, and where
the morphisms are freely generated by commuting morphisms $R\colon
mn\to n$ and $F\colon mn\to m$.

The lower horizontal map in the defining square for $TC$ is a
composite
$$\holim[R,F]\left(THH(A)^{C_n}\right)\prof\to\holim[F]\left(THH(A)^{C_n}\right)\prof\to\holim[F]\left(THH(A)^{hC_n}\right)\prof$$
where the first map is projection to the homotopy limit of the
subcategory generated by the $F$'s only and the second map is the map from fixed points to homotopy fixed points.  The rightmost vertical map is given by the restriction from the homotopy fixed points of all of $\T$ to its finite subgroups. 

This definition is equivalent to Goodwillie's original definition in
terms of an enriched homotopy limit involving a mix of the
restriction, Frobenius and the entire circle action, but is better
suited for our purposes.

\begin{lemma}[Goodwillie]\label{lemma:main}
  For any connective $\ess$-algebra $A$, both the squares in
$$
\begin{CD}
   TC(A)_{(0)}@>>>\left(THH(A)^{h\T}\right)_{(0)}@>>>\left(THH(A)_{(0)}\right)^{h\T}\\
   @VVV@VVV@VVV\\
  TC(A)\prof_{(0)}@>>>\left(THH(A)^{h\T}\right)\prof_{(0)}@>>>\left(THH(A)\prof_{(0)}\right)^{h\T}
\end{CD}
$$
are homotopy cartesian.
\end{lemma}
\begin{proof}
  The right vertical map
  $THH(A)^{h\T}\to(\holim[F]THH(A)^{hC_n})\prof$ in the
  defining square for $TC$ is an equivalence after profinite
  completion (essentially because $\colim{n}BC_n\to B\T$ is a profinite
  equivalence), and so the square
$$
\begin{CD}
  TC(A)@>>>THH(A)^{h\T}\\@VVV@VVV\\
  TC(A)\prof@>>>\left(THH(A)^{h\T}\right)\prof
\end{CD}
$$
is homotopy cartesian even before rationalization.  Both the left and outer square in
$$
\begin{CD}
  THH(A)^{h\T}@>>>\left(THH(A)^{h\T}\right)_{(0)}@>>>\left(THH(A)_{(0)}\right)^{h\T}\\
  @VVV@VVV@VVV\\
  \left(THH(A)^{h\T}\right)\prof@>>>\left(THH(A)^{h\T}\right)\prof_{(0)}@>>>\left(THH(A)\prof_{(0)}\right)^{h\T}
\end{CD}
$$
are homotopy cartesian (they both come from arithmetic squares), and so the right square is homotopy cartesian.  
\end{proof}

A technical issue we are faced with in proving Theorem \ref{thm:main} is commuting homotopy limits and rationalization.  Apart from connectivity arguments we need to be able to commute homotopy $\T$-fixed points and rationalization in the almost free cyclic case.

\begin{lemma}\label{lem:retroffree}
  Given an almost free cyclic spectrum $X$, the map 
$$(\holim[F]X^{hC_n})\prof_{(0)}\to
  (X\prof_{(0)})^{h\T}$$
  is an equivalence.
\end{lemma}
\begin{proof}
  Not using anything about free cyclic spectra, we have that both the
  maps $(\holim[F]X^{hC_n})\prof_{(0)}\to
  \left(X^{h\T}\right)\prof_{(0)}\to
  \left((X\prof)^{h\T}\right)_{(0)}$ are weak equivalences.  Since the
  Tate spectrum vanishes for free cyclic spectra we have that both
  the horizontal $\T$-transfers in 
$$
\begin{CD}
  \Sigma\left((X\prof)_{h\T}\right)_{(0)}@>>>\left((X\prof)^{h\T}\right)_{(0)}\\
  @VVV@VVV\\
  \Sigma(X\prof_{(0)})_{h\T}@>>>(X\prof_{(0)})^{h\T}
\end{CD}
$$
are equivalences, and the Lemma follows since the left vertical map is
an equivalence since homotopy orbits commute with rationalization.
\end{proof}

Let us recall some more or less standard notation.  The category of
finite sets of the form $\mathbf n=\{1,\dots,n\}$ and injections is
denoted $\I$.  We write $S^{\mathbf n}$ for $S^1$ smashed with itself
$n$ times (so that $S^{\mathbf 0}=S^0$). Our $\ess$-algebras $A$ are
either $\Gamma$-spaces or connective symmetric spectra, according to
taste, but ultimately give rise to simplicial functors, and it is as
such they are input to the machinery, and so we write $A(S^{\mathbf
  n})$ for the $n$-th level.  In particular, when $A$ is the
Eilenberg-MacLane spectrum of a simplicial ring $R$, $A(S^{\mathbf
  n})=U(R\otimes \tilde\Z[S^{\mathbf n}])$, where $(\tilde\Z,U)$ is the free/forgetful pair between abelian groups and pointed sets.

In this notation, the $q$-simplices of B\"okstedt's $THH(A)$ is the homotopy colimit over $(\bx_0,\dots,\bx_q)\in\I^{q+1}$ of $Map_*(\bigwedge_{i=0}^qS^{\bx_i},\bigwedge_{i=0}^qA(S^{\bx_i}))$, with Hochschild-style cyclic operators.

Let $\A$ be a square arising as the Eilenberg-MacLane spectra of a split square of simplicial rings and let $I(0)=A^{12}$, $I(1)=\ker\{f^1\}$ and $I(2)=\ker\{f^2\}$.  
For $\bx=(x_0,\dots,x_q)\in\I^{q+1}$, let
$$V^{(k)}(\A)(\bx)=\bigvee_{f}\bigwedge_{i=0}^qI(f(i))(S^{\bx_i})$$  
where the wedge runs over the $f\colon\Z/(q+1)\to\Z/3$ such that $|A_f|=k$, where $A_f$ was defined in \ref{Def:Af}. 
Observe that if $\bx\in\I^{q+1}$ and $\bx^n=(\bx,\dots,\bx)\in\I^{n(q+1)}$ is the diagonal, then
$$V^{(k)}(\A)(\bx^n)^{C_n}\cong
\begin{cases}
  V^{(k/n)}(\A)(\bx)&\text{if $k=0\mod n$}\\
  *&\text{otherwise}
\end{cases}.
$$

 In analogy with the cyclic modules $H(k)$ defined in the proof of Lemma \ref{lem:XPA}, let $T{(k)}$ be the cyclic object whose $q$-simplices is the homotopy colimit over $\bx\in\I^{q+1}$ of $Map_*(\bigwedge_{i=0}^qS^{\bx_i},V^{(k)}(\A)(\bx))$.  We get equivalences of cyclic objects
$$\bigvee_{k>0}T(k)\we\ifi THH(\A)\we \prod_{k>0}T(k),
$$
where the infinite wedge and product are weakly equivalent as the connectivity of $T(k)$ goes to infinity with $k$.

For positive integers $n$ and $k$, let $T(n,k)=sd_nT(k)^{C_n}$, and
extend to rational $n$ and $k$ by setting $T(n,k)=*$ if $n$ or $k$ is not integral.

Restriction induces maps $T(n,k)\to T(n/m,k/m)$ which are interesting
only when $m$ divides both $n$ and $k$.
\begin{lemma}
  The homotopy fiber of the restriction map
  $$T(n,k)\to\holim[m>1]
T(n/m,k/m)\cong\holim[1\neq m|\gcd(n,k)]T(n/m,k/m)
  $$
is equivalent to $T(k)_{hC_n}$.  In particular, if $1=\gcd(n,k)$ we have an equivalence $T(k)_{hC_n}\simeq T(n,k)$
\end{lemma}
\begin{proof}
  This follows by the standard arguments proving the ``fundamental cofibration
sequence'' for fixed points of topological Hochschild homology, as in 
\cite[VI.1.3.8]{DGM}.  For a published account see \cite[5.2.5]{BCD}, but remove the intricacies which are present in the commutative situation where non-cyclic group actions are allowed.

\end{proof}

Consider the homotopy limit of the fixed points of $\prod_{k>0}T(k)$
under the restriction and Frobenius maps.  By prioritizing the
restriction map, we write this as 
$\left(\holim[R]  \prod_{k>0}T(n,k)\right)^{hF}$.  The homotopy limit
of the restriction maps gives the homotopy limit of the diagram
(extended to infinity in both directions)
$$\xymatrix{T(1,1)&T(2,1)&T(3,1)&T(4,1)&T(5,1)&T(6,1)\\
T(1,2)&T(2,2)\ar[ul]&T(3,2)&T(4,2)\ar[ull]&T(5,2)&T(6,2)\ar[ulll]\\
T(1,3)&T(2,3)&T(3,3)\ar@/^1pc/[uull]&T(4,3)&T(5,3)&T(6,3)\ar@/_1pc/[uullll]\\
T(1,4)&T(2,4)\ar[uul]&T(3,4)&T(4,4)\ar@/_1pc/[uull]&T(5,4)&T(6,4)\ar[uulll]\\
T(1,5)&T(2,5)&T(3,5)&T(4,5)&T(5,5)\ar@/_1.5pc/[uuuullll]&T(6,5)\\
T(1,6)&T(2,6)\ar[uuul]&T(3,6)\ar@/^1pc/[uuuull]&T(4,6)\ar[uuull]&T(5,6)&T(6,6)\ar@/_1.8pc/[uuuullll]\ar@/^1.5pc/[uuulll]}
$$
which, by reversal of priorities, is the same as $\holim[R]
\prod_{n>0}T(n,k)$:
\begin{align*}
  \holim[n]\prod_{k>0}T(n,k)&\cong
\holim[n]\prod_{t\in\Q^*}T(n,tn)\\
&\cong
\holim[k]\prod_{t\in\Q^*}T(k/t,k)\cong\holim[k]\prod_{n>0}T(n,k).
\end{align*}

\begin{lemma}\label{lem:ifiTC}Let $\A$ be the square of $\ess$-algebras associated with a split square.
  Then the map
  \begin{align*}
    \ifi TC(\A)\prof_{(0)}\simeq&
\left(\left(\holim[R]  \prod_{k>0}T(n,k)\right)^{hF}\right)\prof_{(0)}\\
\to&\left(\holim[R] \left(\left(\prod_{n>0}T(n,k)\prof\right)_{(0)}\right)\right)^{hF}
  \end{align*}
is an equivalence.
\end{lemma}
\begin{proof}
  If in a tower of spectra the
connectivity of the maps grows to infinity, then the rationalization
of the homotopy limit is equivalent to the homotopy limits of the
rationalized tower.
Since the connectivity of $\prod_{n>0}T(n,k)$ grows to infinity with
$k$ (and the category of natural numbers and factorizations
has cofinal directed subcategories), we have the claimed equivalence.
\end{proof}

\begin{lemma}\label{lem:Rsplit}
  The restriction map 
$$\left(\prod_{n>0}T(n,k)\prof\right)_{(0)}\to \holim[1\neq
l|k]\left(\prod_{n>0}T(n,k/l)\prof\right)_{(0)}$$
is split up to homotopy.
\end{lemma}
\begin{proof}
  We have seen that the homotopy fiber of the restriction map may be identified with
  $\left(\prod_{n>0}T(k)_{hC_n}\prof\right)_{(0)}$, and the lemma
  follows once we know that the left and lower arrows in the commutative diagram
$$\xymatrix{\left({\underset{n>0}{\prod}}T(k)_{hC_n}\right)\prof_{(0)}\ar[r]\ar[d]&
\left({\underset{n>0}{\prod}}T(n,k)\right)\prof_{(0)}\ar[r]&
\left({\underset{n>0}{\prod}}T(k)^{hC_n}\right)\prof_{(0)}\ar[d]\\
\left({\underset{n>0}{\prod}}H((k))_{hC_n}\right)\prof_{(0)}\ar[rr]&&
\left({\underset{n>0}{\prod}}H(H(k))^{hC_n}\right)\prof_{(0)}}
$$
are equivalences.  Here the vertical maps are induced by the linearization maps $T(k)\to H(H(k))$ where $H(k)$ is the cyclic module introduced in the proof of the split
part of Lemma \ref{lem:XPA}. Exactly as in Corollary \ref{cor:tosionorbit} there is a function $L\colon\Z\to\Z_+$ such that $T(k)_{hC}\to H(H(k))_{hC}$ is an $L$-equivalence, and so the infinite product $\left(\prod_{n>0}T(n,k)\right)\to \left(\prod_{n>0}H((k))_{hC_n}\right)$ is also an $L$-equivalence, which shows that the left map in the displayed diagram is an equivalence. The lower map is an equivalence, since the cofiber is $\left(\prod_{n>0}H(H(k))^{tC_n}\right)\prof_{(0)}$, and each Tate homology is $k$-torsion. 
\end{proof}

\begin{cor}\label{lem:identifyingTR}
  The map $$\holim[R] \left(\prod_{n>0}T(n,k)\prof\right)_{(0)}\to \prod_k\left(\prod_{n}H(H(k))^{hC_n}\prof\right)_{(0)}$$
is an equivalence.
 On the right hand side the action by the Frobenius is represented by the product of the maps
 $F\colon H((k))^{hC_{nm}}\to H(H(k))^{hC_m}$ associated to $C_m\subseteq C_{nm}$.
\end{cor}
\begin{proof}
Lemma \ref{lem:Rsplit} gives that the restriction maps split, and so there is an equivalence between $\left(\prod_{n>0}T(n,k)\prof\right)_{(0)}$ and the product of the fibers up to that stage.  We saw in the proof of Lemma \ref{lem:Rsplit} that the map from the fiber $\left({\underset{n>0}{\prod}}T(k)_{hC_n}\right)\prof_{(0)}$ to $\left({\underset{n>0}{\prod}}H(H(k))^{hC_n}\right)\prof_{(0)}$ is a weak equivalence.  Hence the map 
  \begin{align*}
    \left(\prod_{n>0}T(n,k)\prof\right)_{(0)}\to
&\prod_{d|k}\left(\prod_{n}H(H(k/d))^{hC_{n/d}}\prof\right)_{(0)}\\=
&\left(\prod_{n}\prod_{d|\gcd(k,n)}H(H(k/d))^{hC_{n/d}}\prof\right)_{(0)}
  \end{align*}
is a weak equivalence, and the homotopy limit over $R$ just adds successively new factors.
\end{proof}

\begin{cor}\label{cor:end}
All maps in the commuting diagram
$$\xymatrix{{\ifi TC(\A)\prof_{(0)}}\ar[r]\ar[d]&
{\ifi (THH(\A)\prof_{(0)})^{h\T}}\ar[d]\\
{\prod_k\left(\left(\prod_n H(H(k))^{hC_n}\right)^{hF}\right)\prof_{(0)}}\ar[d]&
{\ifi (H(HH(\A))\prof_{(0)})^{h\T}\ar[d]}\\
{\prod_k\left(\left({\holim[F]}H(H(k))^{hC_n}\right)\prof_{(0)}\right)}\ar[r]&
{\prod_k\left(H(H(k))\prof_{(0)}\right)^{h\T}}
}
$$
are equivalences.
\end{cor}
\begin{proof}
The upper left vertical map is an equivalence by the definition of $TC$, Lemma \ref{lem:ifiTC}, Lemma \ref{lem:identifyingTR} and Corollary \ref{lem:identifyingTR}.  
The lower left vertical map is simply rewriting the homotopy limit of a directed system as homotopy fixed points of a product.  
The upper right vertical map  is an equivalence by Lemma \ref{lem:THHvsHH}.
The right lower vertical map is the decomposition of the Hochschild homology of a split square.
The horizontal lower map is an equivalence by Lemma \ref{lem:retroffree} since $H(k)$ is almost free cyclic.
\end{proof}
\begin{proof}[Proof of Theorem \ref{thm:main}]
  As observed in Section \ref{sec:outline}, Theorem \ref{thm:main}
  follows from Lemma \ref{lem:main}, which claims that the cube
  $TC(\A)_{(0)}\to\left(THH(\A)_{(0)}\right)^{h\T}$ is homotopy
  cartesian. 
Lemma \ref{lemma:main} reduces the problem to showing that the cube $TC(\A)\prof_{(0)}\to\left(THH(\A)\prof_{(0)}\right)^{h\T}$ is homotopy
  cartesian. 

Recall from \cite{D97} that we may resolve connective $\ess$-algebras
by simplicial rings.  More precisely, if $A$ is an $\ess$-algebra,
$U\tilde\Z A$ is the $\ess$-algebra obtained by applying the free/forgetful pair
$(\tilde\Z,U)$.  This gives rise to a cosimplicial resolution $A\to\{[q]\to
(U\tilde\Z)^{q+1}A\}$, and the connectivity of
$A\to\holim[q<r](U\tilde\Z)^{q+1}A$ goes to infinity with $r$.  

For our
purposes, it is important to note that if $\A$ is a homotopy
cartesian square, then its underlying cube of spectra is homotopy
cocartesian, and so the cube of ``spectrum homologies'' $U\tilde\Z\A$ is again
homotopy cartesian.  If the maps in $\A$ are $0$-connected, then so
are the maps in $U\tilde\Z\A$.

Furthermore, $U\tilde\Z A$ is naturally equivalent to the
Eilenberg-MacLane spectrum $H(R_A)$, where
$R_A$ is a simplicial ring, and so if $\A$ is a homotopy cartesian
square of $\ess$-algebras, then $R_{\A}$ is a homotopy cartesian
diagram of simplicial rings.

Now, exactly the same set of arguments used in \cite{D97} to reduce the profinite
Goodwillie conjecture to McCarthy's theorem \cite{McC97}, can now be used to see that it is enough to
prove
Lemma \ref{lem:main} in the case where $\A$ the result of applying the
Eilenberg-MacLane functor to a homotopy cartesian square of simplicial
rings and $0$-connected maps.  

By the reduction performed in the proof of Lemma \ref{lem:XPA} it is
enough to consider squares $\A$ associated with split squares of
simplicial rings, and we assume in the rest of the proof that $\A$ has
this form (although all the results used could be generalized to the
more general case using the reductions above).

In this special case the cube $TC(\A)\prof_{(0)}\to\left(THH(\A)\prof_{(0)}\right)^{h\T}$ is homotopy cartesian by Corollary \ref{cor:end}.
\end{proof}

% -----------------

\bibliography{bint}{}

\begin{thebibliography}{10}

\bibitem{BCD}
Morten Brun, Gunnar Carlsson, and Bj{\o}rn~Ian Dundas.
\newblock Covering homology.
\newblock {\em Adv. Math.}, 2010
  \url{http://dx.doi.org/10.1016/j.aim.2010.05.018}.

\bibitem{Cor}
Guillermo Corti{\~n}as.
\newblock The obstruction to excision in {$K$}-theory and in cyclic homology.
\newblock {\em Invent. Math.}, 164(1):143--173, 2006.

\bibitem{CQ}
Joachim Cuntz and Daniel Quillen.
\newblock On excision in periodic cyclic cohomology. {II}. {T}he general case.
\newblock {\em C. R. Acad. Sci. Paris S\'er. I Math.}, 318(1):11--12, 1994.

\bibitem{D97}
Bj{\o}rn~Ian Dundas.
\newblock Relative {$K$}-theory and topological cyclic homology.
\newblock {\em Acta Math.}, 179(2):223--242, 1997.

\bibitem{DGM}
Bj{\o}rn~Ian Dundas, Thomas~G. Goodwillie, and Randy McCarthy.
\newblock The local structure of algebraic k-theory (unpublished).
\newblock 2010.

\bibitem{DK}
Bj{\o}rn~Ian Dundas and Harald~{\O}yen Kittang.
\newblock Excision for {$K$}-theory of connective ring spectra.
\newblock {\em Homology, Homotopy Appl.}, 10(1):29--39, 2008.

\bibitem{GH}
Thomas Geisser and Lars Hesselholt.
\newblock Bi-relative algebraic {$K$}-theory and topological cyclic homology.
\newblock {\em Invent. Math.}, 166(2):359--395, 2006.

\bibitem{Grel}
Thomas~G. Goodwillie.
\newblock Relative algebraic {$K$}-theory and cyclic homology.
\newblock {\em Ann. of Math. (2)}, 124(2):347--402, 1986.

\bibitem{Loday}
Jean-Louis Loday.
\newblock {\em Cyclic homology}, volume 301 of {\em Grundlehren der
  Mathematischen Wissenschaften [Fundamental Principles of Mathematical
  Sciences]}.
\newblock Springer-Verlag, Berlin, second edition, 1998.
\newblock Appendix E by Mar{\'{\i}}a O. Ronco, Chapter 13 by the author in
  collaboration with Teimuraz Pirashvili.

\bibitem{McC97}
Randy McCarthy.
\newblock Relative algebraic {$K$}-theory and topological cyclic homology.
\newblock {\em Acta Math.}, 179(2):197--222, 1997.

\bibitem{MR2182775}
Karl Schwede.
\newblock Gluing schemes and a scheme without closed points.
\newblock In {\em Recent progress in arithmetic and algebraic geometry}, volume
  386 of {\em Contemp. Math.}, pages 157--172. Amer. Math. Soc., Providence,
  RI, 2005.

\bibitem{MR1106918}
R.~W. Thomason and Thomas Trobaugh.
\newblock Higher algebraic {$K$}-theory of schemes and of derived categories.
\newblock In {\em The {G}rothendieck {F}estschrift, {V}ol.\ {III}}, volume~88
  of {\em Progr. Math.}, pages 247--435. Birkh\"auser Boston, Boston, MA, 1990.

\end{thebibliography}
\bibliographystyle{plain}

\bibliographystyle{plain}

\end{document}